\documentclass{amsart}
\usepackage{amsmath}
\usepackage{graphicx}
\usepackage{color}

\usepackage{amssymb}

\usepackage{setspace}
\renewcommand{\emph}{\bf}

\renewcommand{\int}{\operatorname{int}}

\newtheorem{thm}{Theorem}
\newtheorem{corollary}[thm]{Corollary}
\newtheorem{lemma}[thm]{Lemma}

\theoremstyle{definition}
\newtheorem{defn}[thm]{Definition}

\newtheorem{example}[thm]{Example}

\begin{document}

\title{Separable linear orders and universality}
\date{July 20, 2012}
\author{Stefan Geschke}
\address{Hausdorff Center for Mathematics\\Endenicher Allee 62\\53115   Bonn, Germany}
\email{geschke@hcm.uni-bonn.de}
\begin{abstract}
This is mostly a note to myself to clear up some confusion.
\end{abstract}

\maketitle

In various places in the literature, including \cite{GK}, it is stated that every separable linear order embeds into the real line.
This is, however, not the case, at least not with respect to the usual definition
of separability.

\begin{defn}
Let $(L,\leq)$ be a linear order.
$D\subseteq L$ is {\em dense} in $L$ if 
for all $a,b\in L$ with $a<b$ and $(a,b)\not=\emptyset$ there is $d\in D$ with $a<d<b$.
$L$ is {\em separable} if it has a countable dense subset.

Two points $x,y\in L$ form a {\em jump} if $x<y$ and the open interval
$(x,y)$ is empty.
\end{defn}

\begin{lemma}\label{jumps}
Let $L$ be a suborder of $\mathbb R$.
Then $L$ is separable and has only countably many jumps.
\end{lemma}

\begin{proof}
Let $x_0<y_0$ and $x_1<y_1$ be two different jumps in $L$.
Then there are $q_0,q_1\in\mathbb Q$ such that for all $i\in 2$, 
$x_i<q_i<y_i$.
Since the two jumps are different, $y_0\leq x_1$ or $y_1\leq x_0$.
In either case, $q_0\not=q_1$.
It follows that there are not more jumps than rationals, i.e., there are only countably many
jumps.

To see that $L$ is separable, choose a countable set $D\subseteq L$ such that
for all $q\in\mathbb Q$ and all $n>0$ the following holds:
if $L\cap(q-1/n,q+1/n)\not=\emptyset$, then $D\cap(q-1/n,q+1/n)\not=\emptyset$.
It is easily checked that $D$ is dense in $L$.
\end{proof}

\begin{example}
Consider the set $\mathbb R\times 2$ ordered lexicographically.
Then $\mathbb Q\times 2$ is dense in $\mathbb R\times 2$ 
and hence $\mathbb R\times 2$ is separable.
But $\mathbb R\times 2$ has uncountably many jumps and therefore
does not embed into $\mathbb R$.
\end{example}

\begin{thm}\label{universal}
If $L$ is any separable linear order, then there is an order embedding
 $e:L\to\mathbb R\times 2$.
\end{thm}

\begin{proof}
Let $D$ be a countable dense subset of $L$.
We may assume that $D$ contains the first and last element of $L$ provided they exist.
By the saturation of $\mathbb Q$, there is an order embedding $i:D\to\mathbb Q$.
For each $x\in L$ let $$e_1(x)=\sup\{i(d):d\in D\wedge d\leq x\}.$$
Now $e_1:L\to\mathbb R$ preserves $\leq$, but we have $e_1(x)=e_1(y)$ if
$$x>y=\sup\{d\in D:d\leq x\}.$$
Note that this can only happen if $x$ is the successor of $y$ in $L$ and $x\not\in D$.

To correct this failure of injectivity, 
we embed into $\mathbb R\times 2$ rather than $\mathbb R$.
For $x\in L$ let
$$e(x)=\begin{cases}(e_1(x),0)\mbox{, if $x=\sup\{d\in D:d\leq x\}$ and}\\
(e_1(x),1)\mbox{, if $x>\sup\{d\in D:d\leq x\}.$}
\end{cases}$$
\end{proof}

The proof of Theorem \ref{universal} suggests the following notion:

\begin{defn}
Let $L$ be a linear order.
A set $D\subseteq L$ is {\em left dense} if for all 
$x\in L$, $x=\sup\{d\in D:d\leq x\}$.
We define {\em right dense} analogously, using $\inf$ instead of $\sup$.

$L$ is {\em left} ({\em right}) {\em separable} if $L$ has a countable
left (right) dense subset.
\end{defn}

\begin{thm}
For every linear order $L$ the following are equivalent:
\begin{enumerate}
\item $L$ is left separable.
\item $L$ is right separable.
\item $L$ is separable and has only countably many jumps.
\item $L$ order embeds into $\mathbb R$.
\end{enumerate}
\end{thm}

\begin{proof}
If $L$ is left separable, then the map $e_1$ in the proof of Theorem \ref{universal}
embeds $L$ into $\mathbb R$.
This shows the implication from (1) to (4).
The implication from (2) to (4) now follows symmetrically.
If $L$ embeds into $\mathbb R$, then
$L$ is separable and has only countably many jumps by Lemma \ref{jumps}.
Hence (4) implies (3)

If $L$ is separable and has only countably many jumps,
let $D\subseteq L$ be countable, dense, and such that for each jump $x<y$, 
$x,y\in D$.
It is easily checked that $D$ is both left and right dense.
Hence (3) implies both (1) and (2)
\end{proof}

Note that if $x<y$ is a jump of a linear order $L$ and there is an automorphism $\varphi$ of 
$L$ that maps $x$ to $y$, then $y$ and hence $x$ is isolated.
It follows that a separable homogeneous linear order either has no jumps and therefore embeds into
$\mathbb R$ or all its elements are isolated and hence the linear order is isomorphic to the integers 
$\mathbb Z$.
It follows that $\mathbb R$ is universal for homogeneous separable linear orders,
but no universal separable linear order is homogeneous.

Another way to analyze the situation is this:  
given a linear order $L$, we consider two subsets definable without parameters, 
namely the set $J_\ell(L)$ of left partners of a jump and the set $J_r(L)$ of right
partners of a jump.
Also, there is a binary relation that can be defined without parameters, namely
the relation $J(L)$ where $(x,y)\in J(L)$ if $x<y$ or $y<x$ is a jump.

Every automorphism of $L$ preserves $J_\ell(L)$, $J_r(L)$, and the relation $J(L)$.
Therefore it makes sense to define homogeneity as follows:

\begin{defn}
A linear order $L$ is {\em jump homogeneous} if for all finite sets $A,B\subseteq L$
every bijection $b:A\to B$ that preserves the relations $<$, 
$J_\ell(L)$, $J_r(L)$, and $J(L)$
extends to an automorphism of $L$.
\end{defn}

\begin{lemma}
The linear order $\mathbb R\times 2$ is jump homogeneous.
\end{lemma}

\begin{proof}
Let $A,B\subseteq\mathbb R\time 2$ be finite sets and let $f:A\to B$ a bijection that
preserves $<$, $J_\ell(L)$, $J_r(L)$, and $J(L)$.
Note that $J(L)$ is an equivalence relation.
Since $f$ preserves $J(L)$, 
it induces a a bijection $\overline f:A/J(L)\to B/J(L)$.
Since $\mathbb R$ is homogeneous and $(\mathbb R\times 2)/J(L)\cong\mathbb R$, $\overline f$ extends to an automorphism $\overline g$ of
$(\mathbb R\times 2)/J(L)$.
Now $\overline g$ lifts to an automorphism $g$ of $\mathbb R\times 2$.
Since $f$ preserves $J_\ell(L)$ and $J_r(L)$, $g$ extends $f$.
\end{proof}

We now discuss the existence of a universal separable linear order of size 
$\aleph_1<2^{\aleph_0}$.

\begin{defn}
A set $S\subseteq\mathbb R$ is $\aleph_1$-dense if
for all $x,y\in\mathbb R$ with $x<y$, $S\cap(x,y)$ is of size $\aleph_1$.
\end{defn}

Baumgartner showed the following \cite{Baumgartner}:

\begin{thm}
It is consistent with $2^{\aleph_0}=\aleph_2$ that any two $\aleph_1$-dense set of reals 
are order isomorphic.
\end{thm}

\begin{corollary}
It is consistent that there is a universal separable linear order of size 
$\aleph_1<2^{\aleph_0}$.
\end{corollary}

\begin{proof}
By Baumgartner's result, we can assume that any two $\aleph_1$-dense sets of reals
are isomorphic and $\aleph_1<2^{\aleph_0}$.
Let $L$ be a separable linear order of size $\aleph_1$.
By Theorem \ref{universal}, 
$L$ is isomorphic to a subset of $\mathbb R\times 2$.
Let $S\subseteq\mathbb R$ be of size $\aleph_1$ such that $L$ embeds into $S\times 2$.
By enlarging $S$ if necessary, we may assume that $S$ is $\aleph_1$-dense.

It follows that every separable linear order of size $\aleph_1$ embeds into an order
of the form $S\times 2$ where $S$ is an $\aleph_1$-dense subset of $\mathbb R$.
But by our assumption, any two $\aleph_1$-dense subsets of $\mathbb R$ and therefore also any two linear orders of the form $S\times 2$ with $S\subseteq\mathbb R$ $\aleph_1$-dense are isomorphic.
It follows that every linear order of the form $S\times 2$ with $S$ $\aleph_1$-dense
is universal for separable linear orders of size $\aleph_1$.
\end{proof}

\end{document}